\documentclass[twocolumn,amsthm,nameyear]{autart}
\usepackage[T1]{fontenc}
\usepackage{amsmath,amssymb}
\usepackage[authoryear,round]
{natbib}
\usepackage[hidelinks,hypertexnames=false,bookmarks=false]{hyperref}
\hypersetup{pdftitle={Uniform observability criteria for linear time-varying systems},pdfauthor={Soulaimane Berkane}}
\usepackage{xcolor}
\usepackage{tikz}
\usetikzlibrary{positioning}

\journal{Automatica}
\articletype{Paper}
\date{25 October 2026}
\begin{document}
\onecolumn
\thispagestyle{empty}
\begingroup
\centering
{\Large\bfseries LaRSA Technical Report\par}
\vspace{0.5cm}
{\large Robotics and Autonomous Systems Laboratory\par}
{\large Universit\'e du Qu\'ebec en Outaouais\par}
\vspace{1.5cm}
{\huge\bfseries Uniform observability criteria for linear time-varying systems\par}
\vspace{1.2cm}
{\large Soulaimane Berkane\par}
\vspace{1cm}
\begin{tabular}{rl}
    \textbf{Report number:} & LaRSA-TR-2026-02 \\
    \textbf{Date:} & October 2026 \\
    \textbf{Version:} & arXiv submission version 1 \\
    \textbf{Affiliation:} & Universit\'e du Qu\'ebec en Outaouais \\
    \textbf{Laboratory:} & LaRSA -- Robotics and Autonomous Systems Laboratory \\
    \textbf{Contact:} & \href{mailto:soulaimane.berkane@uqo.ca}{soulaimane.berkane@uqo.ca} \\
\end{tabular}
\par\vfill
\begin{minipage}{0.9\textwidth}
\textbf{Abstract.}
This technical report contains the arXiv version of the manuscript entitled
\emph{``Uniform observability criteria for linear time-varying systems''}.
It is circulated as an internal LaRSA report for dissemination and archival purposes.
\end{minipage}
\par\vspace{0.8cm}
\begin{minipage}{0.9\textwidth}
\textbf{Note.}
The pages that follow reproduce the manuscript itself.
\end{minipage}
\par\vfill
{\small \textcopyright\ 2026 Soulaimane Berkane / LaRSA, UQO\par}
\clearpage
\endgroup
\setcounter{page}{1}
\begin{frontmatter}
\title{Uniform observability criteria for linear time-varying systems}
\author[UQO,Lakehead]{Soulaimane Berkane}
\ead{soulaimane.berkane@uqo.ca}
\address[UQO]{Department of Computer Science and Engineering, Universit\'e du Qu\'ebec en Outaouais, Gatineau, QC J8X 3X7, Canada}
\address[Lakehead]{Department of Electrical Engineering, Lakehead University, Thunder Bay, ON P7B 5E1, Canada}
\begin{keyword}
Uniform observability; Linear time-varying systems; Persistence of excitation; Nonlinear observers
\end{keyword}
\begin{abstract}
Verifying uniform observability of a linear time-varying system generally requires its state-transition matrix. This paper develops necessary and sufficient conditions that can be checked directly from the coefficient matrices for two classes of systems. For constant dynamics with real spectrum and a measurable, essentially bounded output matrix, uniform observability is equivalent to persistent excitation of the output matrix on each eigenspace. An equivalent condition requires persistent excitation of the Kalman observability stack. The second result concerns cascades with a constant observable measured subsystem coupled to a hidden state. Under uniform regularity assumptions, linear combinations of output derivatives eliminate the measured state and yield an equivalent observability test for the hidden dynamics. When those dynamics have real spectrum, the test reduces to excitation on each eigenspace intersected with the kernel of the direct measurements. An application to arbitrary-order integrator chains with measurements of ranges to fixed anchors gives an exact excitation condition for the associated lifted LTV system. A second application treats direct hidden measurements and compares the condition with existing sufficient tests.
\end{abstract}
\end{frontmatter}
\raggedbottom

\section{Introduction}

Many nonlinear observer constructions lead to a linear time-varying (LTV) error system whose uniform observability (UO) must be established along the admissible trajectories. This occurs, for example, in direction- and range-aided localization \citep{HS17,BTM21,SBT26}, vision-aided inertial navigation \citep{WBT22}, and attitude and air-velocity estimation using inertial, Pitot, and barometric measurements \citep{TBH25P,TBH26,TBH26P}. Although the observability Gramian gives a precise condition, its evaluation generally requires a transition matrix that is unavailable in closed form. Consequently, applications often rely on structural lemmas that replace the Gramian by a persistence-of-excitation (PE) condition on a measured or reconstructed signal.

Uniform complete observability already appears in the linear filtering theory of \citet{KB61}. Its role in Riccati equations and time-varying stability was subsequently developed by \citet{Bucy67} and \citet{AM69}. Classical linear observability theory is presented by \citet{Chen84}, while \citet{ASP98} study observability conditions for Lyapunov stability of time-varying systems. For observer design, the practical difficulty is to obtain conditions that retain the uniform-in-time requirement while involving signals and matrices that can be checked in a given model.

Tests expressed through the coefficient matrices and output derivatives have a classical foundation in \citet{SM67}. Under suitable boundedness and regularity assumptions, \citet[Theorem~4]{BPP10} establish UO from a uniform positive lower bound on the Gram matrix of a finite output-derivative stack at every instant. \citet[Theorem~3]{BPSO17} relax this requirement to an instant in every fixed-length time window, assuming in particular that the highest derivative block is Lipschitz continuous. Averaged determinant conditions in \citet{Scandaroli13} and \citet[Proposition~1]{MES17} also permit temporary rank loss; the latter uses a uniform positive window integral of the absolute determinant of a square matrix formed from selected derivative rows. A rectangular formulation and a transfer result from derivative outputs to the original output are given in \citet{HHS20}.

The relation between excitation and dynamics has also been studied in identification and learning control. \citet{MG88} establish conditions for transferring PE between regression vectors related by linear dynamical transformations, including time-varying systems. \citet{MHM92} connect excitation to uniform complete observability through measured dynamics under regularity assumptions. More recently, \citet{LH25} study uniform observability with respect to invariant subspaces and its preservation under time-varying output feedback, with applications to synchronization.

Several observability criteria exploit additional structure in the dynamics and the output. The real-spectrum criterion in \citet[Lemma~2.7]{HS17} treats a constant observable pair with a persistently exciting output weighting. For the translational integrator chain considered in \citet[Lemma~1]{BTM21}, UO is already characterized by PE of the position-output matrix. The nilpotent-component criterion in \citet[Lemma~1]{WBT22} exploits more of the state matrix through a stack of output rows multiplied by powers of its nilpotent part. Restricted-image conditions for nilpotent dynamics appear in \citet[Lemma~3]{SBT26}. A complementary argument differentiates a measured integrator-chain output until a lower-dimensional hidden state becomes visible; this is used in \citet[Propositions~6--7]{TBD21} and, with direct hidden measurements, in \citet[Lemma~2]{TBH26}. Finite-dimensional LTV immersions for linear dynamics with quadratic outputs are developed in \citet{TBD21,BTHD23}; the latter also treats multiple quadratic measurements, while UO of the extended system depends on the input trajectory. For several measured outputs, linear combinations of their derivatives can eliminate a state through a left nullspace of an observability matrix, as in continuous-time parity relations \citep[Section~III]{Varga09}.

This paper develops necessary and sufficient observability conditions for two structured classes of LTV systems. The analysis identifies the state directions that require persistent excitation and establishes when eliminating a measured subsystem preserves UO in both directions.

The first contribution is a necessary and sufficient condition for constant state matrices with real spectrum and measurable, essentially bounded output matrices: UO is equivalent to PE of the output matrix on each eigenspace separately. The subspace conditions thus characterize observability of the full state. An equivalent averaged Kalman test uses the full stack $\operatorname{col}(C,CA,\ldots,CA^{n-1})$, without output derivatives or a prescribed factorization. For $C(t)=\Theta(t)H$ with measurable, essentially bounded $\Theta$ and an observable constant pair $(A,H)$, the exact condition is PE of $\Theta$ on each eigenspace image under $H$. This can be weaker than the full measurement-space excitation condition in \citet{HS17}. It also distinguishes the exact condition from PE of the full nilpotent-component stack in \citet{WBT22}, which requires excitation on the sum of all eigenspaces rather than on each separately.

The second contribution is an exact observability reduction for a constant observable measured subsystem coupled to hidden constant dynamics. Linear combinations of a finite stack of output derivatives eliminate the measured state. Under explicit uniform regularity assumptions, UO of the resulting hidden pair is equivalent to UO of the full cascade. The reduction permits several measured outputs and direct hidden measurements, and requires no spectral restriction. Its scalar-chain specialization supplies a converse to the sufficient reduction in \citet{TBD21}. When the hidden dynamics have real spectrum, the first characterization further restricts the required excitation to each hidden eigenspace intersected with the kernel of the direct measurement matrix.

Applications to arbitrary-order integrator chains with single or multiple range measurements and to a second-order cascade illustrate how the exact subspace requirement can be weaker than full-space PE. These conclusions concern LTV observability; nonlinear observer stability requires additional hypotheses.

Section~2 recalls the observability and derivative criteria used in the analysis. Section~3 establishes the eigenspace characterization, develops its Kalman-stack and factored-output formulations, and compares them with existing sufficient criteria. Section~4 develops the cascade reduction and illustrates direct hidden measurements. Section~5 combines both main results in a range-based state-estimation case study. The longer proofs are collected in Appendices~\ref{app:proof-real}--\ref{app:proof-cascade}.

\noindent\textbf{Notation.}
The sets of real numbers, real $n$-dimensional column vectors, and real $m\times n$ matrices are denoted by $\mathbb R$, $\mathbb R^n$, and $\mathbb R^{m\times n}$, respectively. All vector norms are Euclidean and matrix norms are induced. The transpose of a matrix is denoted by $(\cdot)^\top$, and $I_n$ is the $n\times n$ identity matrix. For symmetric matrices of the same size, $Q\succeq P$ and $Q\succ P$ mean that $Q-P$ is positive semidefinite and positive definite, respectively. For a symmetric positive semidefinite matrix $Q$, the notation $Q^{1/2}$ denotes its unique symmetric positive semidefinite square root: the matrix $R=R^\top\succeq0$ satisfying $R^2=Q$. This root exists even when $Q$ is singular and is obtained by taking the nonnegative square roots of its eigenvalues in an orthogonal diagonalization. The symbols $\ker M$, $\operatorname{im}M$, $\operatorname{rank}M$, and $\operatorname{spec}(M)$ denote its kernel, image, rank, and spectrum, respectively. The notation $\operatorname{col}(M_1,\ldots,M_k)$ denotes vertical concatenation, and $\operatorname{span}(v_1,\ldots,v_k)$ denotes the linear span. We write $e_i$ for the $i$th canonical basis vector, with dimension determined by the context, and $\mathbb S^{n-1}=\{x\in\mathbb R^n:\|x\|=1\}$. A basis matrix for a subspace has orthonormal columns unless stated otherwise.

A matrix-valued function $F$ is called measurable when each entry $F_{ij}$ is Lebesgue measurable; equivalently, $\{t\geq0:F_{ij}(t)>a\}$ is a Lebesgue-measurable set for every $i,j$ and $a\in\mathbb R$. All time integrals are Lebesgue integrals. A property holds almost everywhere if it holds outside a set of Lebesgue measure zero. The essential supremum $\operatorname*{ess\,sup}_{t\geq0}\|f(t)\|$ is the smallest bound on $\|f(t)\|$ that holds almost everywhere; $f$ is essentially bounded when this quantity is finite.

For an integer $k\geq0$, $\mathcal C_b^k$ denotes the class of functions on $[0,\infty)$ whose derivatives of orders $0,\ldots,k$ are continuous and uniformly bounded; the subscript $b$ indicates boundedness. This definition applies componentwise to matrix-valued functions. For an interval $I$, we use $\|f\|_{\mathcal L^2(I)}=(\int_I\|f(s)\|^2ds)^{1/2}$ and $\|f\|_{\mathcal L^1(I)}=\int_I\|f(s)\|ds$.

\section{Uniform observability and standard tools}

Consider the linear time-varying (LTV) system
\begin{equation}
\dot x=A(t)x,\qquad y=C(t)x,\qquad
x\in\mathbb R^n,\quad y\in\mathbb R^p,
\label{eq:ltv}
\end{equation}
where \(A(t)\) and \(C(t)\) are measurable and essentially bounded, unless stronger assumptions are explicitly stated. Solutions of the state equation are understood to be locally absolutely continuous and to satisfy the equation almost everywhere. Let \(\Phi(s,t)\) denote the state-transition matrix from time \(t\) to time \(s\). It satisfies, for almost every $s$,

\[
\frac{\partial}{\partial s}\Phi(s,t)=A(s)\Phi(s,t),
\qquad
\Phi(t,t)=I_n.
\]
Following the convention of \citet{HS17}, we use the normalized observability Gramian
\begin{equation}
 W_\delta(t):=\frac1\delta\int_t^{t+\delta}
 \Phi(s,t)^\top C(s)^\top C(s)\Phi(s,t)\,ds.
 \label{eq:gramian}
\end{equation}
\begin{defn}[Uniform observability]\label{def:uo}
The pair $(A,C)$ is \emph{uniformly observable} (UO) if there exist $\delta,\mu>0$ such that $W_\delta(t)\succeq\mu I_n$ for every $t\geq0$.
\end{defn}
For the trajectory starting from $x(t)=x_0$, this means
\[
 \frac1\delta\int_t^{t+\delta}\|y(s)\|^2ds
 =x_0^\top W_\delta(t)x_0\geq\mu\|x_0\|^2.
\]
Thus the integral of the squared output norm must have a lower bound proportional to the squared initial-state norm, uniformly in the starting time.

The same constants $\delta$ and $\mu$ must work for every initial time $t\geq0$. Under the bounded-coefficient assumptions, $W_\delta(t)$ also has a uniform upper bound for every fixed $\delta>0$. Thus Definition~\ref{def:uo} is equivalent to the usual two-sided Gramian condition for uniform complete observability \citep{KB61,AM69}.

Other notions of uniform observability are used in nonlinear observer theory \citep{GK94,Besancon07}. Throughout this paper, UO refers to the uniform-in-time LTV Gramian condition in Definition~\ref{def:uo}.

\begin{rem}[Gramian normalization]
The factor $1/\delta$ makes $x_0^\top W_\delta(t)x_0$ the average squared output norm over the observation window, so the excitation bound is expressed per unit time. The unnormalized Gramian is $\widetilde W_\delta(t):=\delta W_\delta(t)$. Hence the condition $\widetilde W_\delta(t)\succeq\alpha I_n$ for every $t\geq0$ is equivalent to Definition~\ref{def:uo}, with $\alpha=\delta\mu$. The admissible lower bound may depend on the chosen window length $\delta$.
\end{rem}

\begin{rem}[Weighted observability]
For a bounded measurable symmetric measurement weight $Q:[0,\infty)\to\mathbb R^{p\times p}$ with $Q(t)\succeq0$, define
\[
 W_{\delta,Q}(t):=\frac1\delta\int_t^{t+\delta}
 \Phi(s,t)^\top C(s)^\top Q(s)C(s)\Phi(s,t)\,ds.
\]
The positive semidefinite square-root map is continuous, so the matrix function $Q^{1/2}$ is measurable and bounded, and $W_{\delta,Q}$ is precisely the observability Gramian of $(A,Q^{1/2}C)$. Thus a uniform positive lower bound on $W_{\delta,Q}$ is equivalent to UO of this pair. If $q_-I_p\preceq Q(t)\preceq q_+I_p$ for constants $q_-,q_+>0$, it is equivalent to UO of $(A,C)$. A merely positive semidefinite weight is also admissible, provided that the weighted Gramian itself satisfies the lower bound; unweighted UO alone does not ensure it. This distinction is used in the Riccati framework of \citet[Lemma~2.6]{HS17}, which also imposes a condition on the process-weight Gramian. Weighted formulations also appear in \citet[Definition~2]{Bucy67}.
\end{rem}

\begin{rem}[Riccati observers]
Measurable coefficients are compatible with the Kalman--Bucy framework; see \citet[Section~2]{AC19} and \citet[Theorem~1]{FTR08}. Riccati solutions are then locally absolutely continuous and satisfy the equation almost everywhere. For the observer based on
\[
 \dot P=AP+PA^\top+V-PC^\top QCP,\qquad P(0)\succ0,
\]
UO together with bounded measurable symmetric uniformly positive definite weights $Q,V$ ensures uniform Riccati bounds and exponential stability of the homogeneous error $\dot e=(A-PC^\top QC)e$; see \citet[Lemma~2.6]{HS17}. Corrected bounds for solutions of the continuous Riccati equation under uniform weighted observability and controllability conditions are given by \citet{PRV01}. The observability criteria below address the measurement condition in this result; observer convergence also depends on the weight and model assumptions.
\end{rem}

Successive output derivatives provide standard tests when the transition matrix is not available; their classical differential formulation is developed in \citet{SM67}. For an integer $K\geq1$ and coefficients $A\in\mathcal C_b^K$, $C\in\mathcal C_b^{K+1}$, define the derivative stack by
\begin{equation}
 \begin{aligned}
 N_0(t)&:=C(t),\\
 N_{k+1}(t)&:=\dot N_k(t)+N_k(t)A(t),\quad 0\leq k<K,\\
 \mathcal O_K(t)&:=\operatorname{col}(N_0(t),\ldots,N_K(t)).
 \end{aligned}
 \label{eq:derivative-stack}
\end{equation}
Under this regularity, the $k$th output derivative along \eqref{eq:ltv} is $N_kx$, and $N_0,\ldots,N_K$ and their first time derivatives are uniformly bounded.

The identity $\operatorname{col}(y,\dot y,\ldots,y^{(K)})=\mathcal O_K(t)x$ relates the derivative stack to the state. A bound $\mathcal O_K(t)^\top\mathcal O_K(t)\succeq\alpha I_n$ means that $\sum_{k=0}^K\|N_k(t)x\|^2\geq\alpha\|x\|^2$ for all $t$ and $x$. Under the stated regularity, this gives the following sufficient condition for UO \citep[Theorem~4]{BPP10}.
\begin{lem}[Uniform derivative-rank criterion; {\citealp{BPP10}}]\label{lem:pointwise-rank}
Let $K\geq1$ be an integer, assume $A\in\mathcal C_b^K$ and $C\in\mathcal C_b^{K+1}$, and define $\mathcal O_K$ by \eqref{eq:derivative-stack}. If there exists $\alpha>0$ such that
\[
 \mathcal O_K(t)^\top\mathcal O_K(t)
 =\sum_{k=0}^K N_k(t)^\top N_k(t)\succeq\alpha I_n,
 \qquad t\geq0,
\]
then $(A,C)$ is UO.
\end{lem}

Requiring the derivative stack to distinguish all directions at every instant can be restrictive when measurements vary with time. The next condition permits temporary rank loss: it requires the determinant of the Gram matrix of selected derivative rows to have a positive average on every observation window. This criterion appears in \citet[Lemma~3.1 and Appendix~A.1]{Scandaroli13}; \citet[Proposition~1]{MES17} give a square-matrix formulation. We use the rectangular version stated in \citet[Proposition~2.4, C1]{HHS20}.
\begin{lem}[Averaged determinant criterion; {\citealp{Scandaroli13,MES17}}]\label{lem:determinant}
Let $K\geq1$ be an integer, assume $A\in\mathcal C_b^K$ and $C\in\mathcal C_b^{K+1}$, and define $\mathcal O_K$ by \eqref{eq:derivative-stack}. Let $M(t)\in\mathbb R^{r\times n}$, $r\geq n$, consist of a fixed selection of rows of $\mathcal O_K(t)$. If there exist $\delta,\eta>0$ such that
\[
 \frac1\delta\int_t^{t+\delta}\det\bigl(M(s)^\top M(s)\bigr)ds
 \geq\eta,\qquad t\geq0,
\]
then $(A,C)$ is UO.
\end{lem}
For square $M$, this is equivalent to the averaged $|\det M|$ condition in \citet{MES17}.
The derivative rows can also be treated as an output matrix in their own right. Under the regularity assumed above, UO with these derivative outputs implies UO with the original output. This observation, used in the proof of \citet{MES17} and stated explicitly by \citet{HHS20}, provides a way to apply other observability tests to selected derivative rows.
\begin{lem}[Observability from output derivatives; {\citealp[Lemma~2.3]{HHS20}}]\label{lem:transfer}
Let $K\geq1$ be an integer, assume $A\in\mathcal C_b^K$ and $C\in\mathcal C_b^{K+1}$, and define $\mathcal O_K$ by \eqref{eq:derivative-stack}. Let $M(t)$ consist of a fixed selection of rows of $\mathcal O_K(t)$. If $(A,M)$ is UO, then $(A,C)$ is UO.
\end{lem}

If $A$ is constant with real spectrum, the preceding lemma can be combined with the result of \citet[Lemma~2.7]{HS17}, with its constant sensing matrix equal to $I_n$, to obtain an averaged matrix test for the selected derivative rows.
\begin{cor}[Averaged derivative-output criterion; {\citealp[Proposition~2.4, C2]{HHS20}}]\label{cor:derivative-average}
Let $K\geq1$ be an integer, let $A\in\mathbb R^{n\times n}$ be constant with real spectrum, and assume $C\in\mathcal C_b^{K+1}$. Define $\mathcal O_K$ by \eqref{eq:derivative-stack}. If a fixed selection $M$ of rows of $\mathcal O_K$ satisfies
\[
 \frac1\delta\int_t^{t+\delta}M(s)^\top M(s)ds
 \succeq\alpha I_n,\qquad t\geq0,
\]
for some $\delta,\alpha>0$, then $(A,C)$ is UO.
\end{cor}
This averaged matrix condition permits excitation to accumulate in different directions at different times. The determinant condition instead requires full column rank on a sufficient part of each window. The next section introduces persistence of excitation to describe such averaged inequalities and identifies the exact subspaces on which excitation is required for constant real-spectrum dynamics.

\section{Constant dynamics with real spectrum}

When the dynamics are constant, the restriction of the state transition to an eigenspace is a scalar exponential. This observation suggests examining how the output matrix measures fixed vectors in each eigenspace. We express the required uniform lower bounds through persistence of excitation.

\begin{defn}[Persistence of excitation on a subspace]\label{def:pe}
Let $E\subseteq\mathbb R^n$ be a fixed linear subspace. A measurable, essentially bounded matrix function $F:[0,\infty)\to\mathbb R^{r\times n}$ is \emph{persistently exciting on $E$}, or \emph{PE on $E$}, if there exist $\delta,\mu>0$ such that
\begin{equation}
 \frac1\delta\int_t^{t+\delta}\|F(s)v\|^2\,ds
 \geq\mu\|v\|^2,\qquad t\geq0,\quad v\in E.
 \label{eq:pe}
\end{equation}
\end{defn}
If $V\in\mathbb R^{n\times d}$ has orthonormal columns spanning $E$, where $d=\dim E>0$, condition~\eqref{eq:pe} is equivalent to
\begin{equation}
 \frac1\delta\int_t^{t+\delta}V^\top F(s)^\top F(s)V\,ds
 \succeq\mu I_d,\qquad t\geq0.
 \label{eq:pebasis}
\end{equation}
Indeed, writing $v=Vw$ gives $\|v\|=\|w\|$ and converts \eqref{eq:pe} into the quadratic-form inequality associated with \eqref{eq:pebasis}. The condition is vacuous for $E=\{0\}$. Unless a subspace is specified, PE means PE on the full domain of $F$. The constants must be uniform over both the initial time and the vectors in the subspace. In particular, the vector $v$ in \eqref{eq:pe} is fixed throughout the integration interval, whereas the observability Gramian measures its evolution under the state transition.

Subspace formulations of excitation also arise in adaptive control; see, for example, \citet{MUB25}. Definition~\ref{def:pe} is the usual PE condition applied to $FV$, where $V$ is a basis matrix for the chosen subspace. This use of subspaces differs from the uniform observability with respect to invariant subspaces studied by \citet{LH25}. Here the objective remains UO on the full state space: the subspaces specify which directions of the output matrix must be persistently excited to characterize that property.

\begin{rem}[PE does not by itself imply UO]\label{rem:rotation}
Take $A=\left[\begin{smallmatrix}0&-1\\1&0\end{smallmatrix}\right]$ and $C(t)=[\cos t,\sin t]$. The average of $C^\top C$ on every window of length $\pi$ is $I_2/2$, so $C$ is PE. However, $C(t+s)e^{As}=[\cos t,\sin t]$, and every observability Gramian has rank one. Thus excitation of fixed vectors can be canceled by the state dynamics. A related rotation counterexample is discussed in \citet[Section~4.3]{BPSO17}.
\end{rem}

Assume in this section that $A$ is constant and all its eigenvalues are real. For each $\lambda\in\operatorname{spec}(A)$, let $E_\lambda:=\ker(A-\lambda I_n)$. Since $e^{As}v=e^{\lambda s}v$ for $v\in E_\lambda$, UO necessarily requires PE of $C$ on each $E_\lambda$. The substantive point is that these conditions also suffice for the full state, even when the eigenspaces do not span it. The dynamics reveal the generalized-eigenvector directions, so no additional excitation test on the generalized eigenspaces is needed. The following theorem establishes this conclusion under the same bounded measurable output assumptions used in the Gramian definition.

\begin{thm}[Eigenspace characterization]\label{thm:real}
Let $A\in\mathbb R^{n\times n}$ be constant with $\operatorname{spec}(A)\subset\mathbb R$, and let $C:[0,\infty)\to\mathbb R^{p\times n}$ be measurable and essentially bounded. The pair $(A,C(t))$ is UO if and only if $C$ is PE on $E_\lambda=\ker(A-\lambda I_n)$ for every $\lambda\in\operatorname{spec}(A)$.
\end{thm}

The proof is given in Appendix~\ref{app:proof-real}.

Theorem~\ref{thm:real} and \eqref{eq:pebasis} express UO through averaged quadratic forms on fixed eigenspace bases. Unlike the derivative-output criteria in Lemmas~\ref{lem:transfer} and~\ref{lem:determinant}, this characterization requires only measurability and essential boundedness of $C$.

\begin{rem}[Scope of the hypotheses]
The real-spectrum assumption excludes the cancellation in Remark~\ref{rem:rotation}. Boundedness of $C$ is also essential. Take $A=\left[\begin{smallmatrix}0&1\\0&0\end{smallmatrix}\right]$ and $C(t)=[1,-t]$. The restriction of $C$ to $\ker A=\operatorname{span}(e_1)$ is PE, but $C$ is unbounded and $C(t+s)e^{As}=[1,-t]$. Every Gramian has rank one. In the proof in Appendix~\ref{app:proof-real}, boundedness controls the remainder in the asymptotic state expansion uniformly in the initial time.
\end{rem}

\subsection{An averaged Kalman test}

The eigenspace characterization isolates the directions that must be excited. It can be checked through \eqref{eq:pebasis} using a basis for each eigenspace. An equivalent formulation can be obtained directly from the powers of $A$. For constant outputs, these products form the familiar Kalman observability matrix. When $C$ varies with time, the corresponding condition is a uniform positive lower bound on the average of $\sum_{k=0}^{n-1}(A^k)^\top C(t)^\top C(t)A^k$. Its entries involve only $C(t)$ and $A$, so this construction does not require the output derivatives used in Section~2.

\begin{cor}[Averaged Kalman characterization]\label{cor:apower}
Let $A\in\mathbb R^{n\times n}$ be constant with real spectrum, and let $C:[0,\infty)\to\mathbb R^{p\times n}$ be measurable and essentially bounded. Define
\[
 \mathcal O_A(t):=\operatorname{col}
 \bigl(C(t),C(t)A,\ldots,C(t)A^{n-1}\bigr).
\]
Then $(A,C(t))$ is UO if and only if $\mathcal O_A$ is PE.
\end{cor}

The proof is given in Appendix~\ref{app:proof-apower}.

When $C$ is constant, PE of $\mathcal O_A$ reduces to $\operatorname{rank}\mathcal O_A=n$, the usual Kalman condition. For a time-varying $C$, the averaged condition allows $\mathcal O_A(t)$ to lose rank at some times. The spectral hypothesis remains essential: in Remark~\ref{rem:rotation}, $\mathcal O_A(t)^\top\mathcal O_A(t)=I_2$ at every time, although the pair is not UO.

\subsection{Factored outputs}

An orthogonal change of coordinates reduces the tilt and air-velocity subsystem in \citet[Lemma~1]{TBH26P} to constant nilpotent dynamics with a time-varying output, providing an application of the real-spectrum criterion. In many such models, the output separates into a constant sensing matrix and a time-varying measurement factor, $C(t)=\Theta(t)H$. The matrix $H$ maps a state direction $v\in E_\lambda$ to the measurement-space direction $Hv$, which is then weighted by $\Theta(t)$. Observability of $(A,H)$ ensures that this map loses no nonzero direction within an eigenspace. The excitation requirement can therefore be expressed on the images $HE_\lambda=\{Hv:v\in E_\lambda\}$, without requiring a lower bound on the whole domain of $\Theta$.

\begin{cor}[Factored outputs with real spectrum]\label{cor:factorized-real}
Let $A\in\mathbb R^{n\times n}$ and $H\in\mathbb R^{r\times n}$ be constant, with real spectrum of $A$ and observable pair $(A,H)$. Let $\Theta:[0,\infty)\to\mathbb R^{p\times r}$ be measurable and essentially bounded, and set $E_\lambda:=\ker(A-\lambda I_n)$. Then $(A,\Theta(t)H)$ is UO if and only if $\Theta$ is PE on $HE_\lambda$ for every $\lambda\in\operatorname{spec}(A)$.
\end{cor}
\begin{proof}
If $v\in E_\lambda$ and $Hv=0$, then $HA^kv=\lambda^kHv=0$ for every $k$. Observability of $(A,H)$ therefore implies $v=0$. Thus $H$ is injective on $E_\lambda$, and there exists $\sigma_\lambda>0$ such that $\|Hv\|\geq\sigma_\lambda\|v\|$ for all $v\in E_\lambda$.

If $\Theta$ is PE on $HE_\lambda$ with constants $\delta,\nu>0$, then
\[
 \frac1\delta\int_t^{t+\delta}\|\Theta(s)Hv\|^2ds
 \geq\nu\sigma_\lambda^2\|v\|^2.
\]
Conversely, if $\Theta H$ is PE on $E_\lambda$ with constants $\delta,\mu>0$, write $u=Hv$ for $u\in HE_\lambda$ and $v\in E_\lambda$. Since $\|u\|\leq\|H\|\|v\|$ and $\|H\|>0$,
\[
 \frac1\delta\int_t^{t+\delta}\|\Theta(s)u\|^2ds
 \geq\frac{\mu}{\|H\|^2}\|u\|^2.
\]
The two subspace conditions are equivalent, and Theorem~\ref{thm:real} applies.
\end{proof}

Figure~\ref{fig:constant-criteria} collects the equivalent tests for constant real-spectrum dynamics and their factored-output formulation.
\begin{figure}[tb]
\centering
\begin{tikzpicture}[
 criterion/.style={draw,rounded corners=1pt,align=center,
   text width=.85\columnwidth,inner sep=4pt,font=\small},
 node distance=6mm]
 \node[criterion] (uo) {$(A,C(t))$ uniformly observable};
 \node[criterion,below=of uo] (eig)
   {$C$ PE on each $E_\lambda=\ker(A-\lambda I_n)$};
 \node[criterion,below=of eig] (kalman)
   {$\operatorname{col}(C,CA,\ldots,CA^{n-1})$ PE};
 \path (uo.south) -- node {$\Updownarrow$} (eig.north);
 \path (eig.south) -- node {$\Updownarrow$} (kalman.north);
 \node[criterion,below=5mm of kalman] (factored)
   {For $C=\Theta H$ and observable $(A,H)$:\\[2pt]
    UO $\Longleftrightarrow$ $\Theta$ PE on each $HE_\lambda$};
\end{tikzpicture}
\caption{Constant-dynamics tests (Theorem~\ref{thm:real} and Corollaries~\ref{cor:apower} and~\ref{cor:factorized-real}). The matrix $A$ is constant with real spectrum; $C$ is measurable and essentially bounded. For the factored statement, $H$ is constant, $(A,H)$ is observable, and $\Theta$ is measurable and essentially bounded. Each eigenspace is tested separately.}
\label{fig:constant-criteria}
\end{figure}
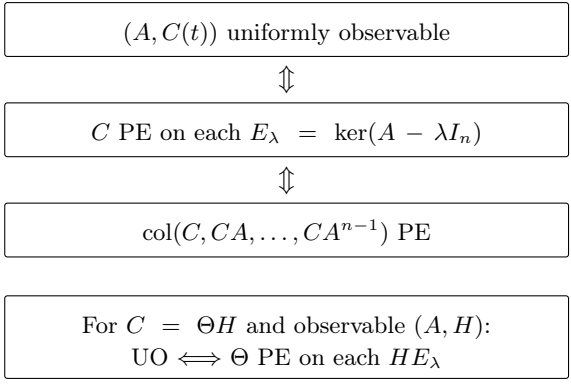

Excitation on the entire measurement space implies excitation on each image $HE_\lambda$. This recovers the sufficient weighting criterion of \citet[Lemma~2.7]{HS17}, written here in the same factorization notation.

\begin{cor}[Full measurement-space excitation]\label{cor:hs17}
Let $A\in\mathbb R^{n\times n}$ and $H\in\mathbb R^{r\times n}$ be constant, with real spectrum of $A$ and observable pair $(A,H)$. Let $\Theta:[0,\infty)\to\mathbb R^{p\times r}$ be measurable and essentially bounded. If $\Theta$ is PE on $\mathbb R^r$, then $(A,\Theta(t)H)$ is UO.
\end{cor}
\begin{proof}
PE on $\mathbb R^r$ implies PE on each $HE_\lambda$, so Corollary~\ref{cor:factorized-real} applies. Taking the weighting matrix in \citet{HS17} to be $\Theta(t)^\top\Theta(t)$ gives the same observability Gramian.
\end{proof}

The sufficient condition requires a uniform lower bound on the average of $\Theta^\top\Theta$ over all of $\mathbb R^r$. The exact condition requires that bound only on each $HE_\lambda$ separately. The difference persists even for smooth outputs and can be independent of the chosen constant factorization.

\begin{exmp}[Separate eigenspaces versus full-space PE]\label{ex:plateaus}
Take $A=\operatorname{diag}(0,1)$, $H=I_2$, and
\[
 \Theta(t)=C(t)=[1,1+e^{-t}].
\]
The output matrix is smooth and bounded, with bounded derivatives of every order. Its restrictions to $E_0=\operatorname{span}(e_1)$ and $E_1=\operatorname{span}(e_2)$ satisfy $|C(t)ae_1|^2=a^2$ and $|C(t)ae_2|^2\geq a^2$. Thus $C$ is PE on both eigenspaces, and Theorem~\ref{thm:real} proves UO.

In contrast, for $w=(-1,1)^\top$ one has $C(t)w=e^{-t}$. For every fixed $\delta>0$,
\[
 \frac1\delta\int_t^{t+\delta}|C(s)w|^2ds
 \leq e^{-2t}\longrightarrow0.
\]
Hence $C$ is not PE on $\mathbb R^2$. The decay of $C(t)w$ concerns a fixed vector; the measured trajectory is instead $C(t+s)e^{As}x=x_1+(e^s+e^{-t})x_2$, in which the distinct exponential factors separate the state components.

Now consider any other constant factorization $C(t)=\widetilde\Theta(t)\widetilde H$, with $\widetilde H\in\mathbb R^{r\times2}$. The rows $C(0)$ and $C(1)$ are independent and belong to the row space of $\widetilde H$, so $\widetilde H$ has column rank two. If $\widetilde\Theta$ were PE on $\mathbb R^r$ with constants $\delta,\nu>0$, it would follow that
\[
 \frac1\delta\int_t^{t+\delta}C(s)^\top C(s)ds
 \succeq\nu\widetilde H^\top\widetilde H\succ0
\]
uniformly in $t$, contradicting the preceding estimate along $w$. Thus no constant factorization satisfies the full measurement-space excitation condition.
\end{exmp}

\begin{rem}[Nilpotent specialization]\label{rem:nilpotent}
If $A$ is nilpotent, its only eigenvalue is zero and $E_0=\ker A$. Under the measurable, essentially bounded output assumptions of Theorem~\ref{thm:real}, UO is therefore equivalent to PE of $C$ on $\ker A$. Likewise, under the hypotheses of Corollary~\ref{cor:factorized-real}, the factored-output condition becomes PE of $\Theta$ on $H\ker A$. Both conclusions follow by setting $\lambda=0$ in the preceding results.

This specialization also explains the restricted-image condition in \citet[Lemma~3]{SBT26}. For nilpotent $A$ and observable $(A,H)$, put $L_k:=\bigcap_{i\geq k+1}\ker(HA^i)$, $k=0,\ldots,n-1$. If $x\in L_k$, then $HA^j(A^{k+1}x)=0$ for every $j\geq0$, so observability gives $A^{k+1}x=0$. The reverse inclusion is immediate; hence $L_k=\ker A^{k+1}$ and $HA^kL_k\subseteq H\ker A$, with equality at $k=0$. Requiring $\Theta$ to be PE on every image is thus equivalent to the single condition on $H\ker A$.
\end{rem}

\begin{exmp}[A redundant sensor in a double integrator]\label{ex:integrator}
Consider $\dot p=v$, $\dot v=0$, with $p,v\in\mathbb R^d$ and outputs $y_1=\Theta_1(t)p$, $y_2=\Theta_2(t)v$, where both sensor matrices are measurable and essentially bounded. The state matrix is nilpotent and its kernel consists of $(p,0)$. Theorem~\ref{thm:real} with $\lambda=0$ therefore shows that the system is UO if and only if $\Theta_1$ is PE. No excitation is required of $\Theta_2$, which may vanish identically. The averaged matrix test has dimension $d$ instead of $2d$.

Applying Corollary~\ref{cor:hs17} to the factorization with $H=I_{2d}$ would require excitation in both sensor blocks. This comparison depends on the factorization: discarding the velocity output and choosing the constant position matrix also permits that criterion to establish UO. The kernel condition identifies the relevant state directions directly.
\end{exmp}

\subsection{Relation to the criterion of Wang et al.}

The eigenspace and averaged Kalman characterizations also recover the sufficient criterion of \citet[Lemma~1]{WBT22}. To state it, write the real Jordan--Chevalley decomposition $A=S+J$, where $S$ is diagonalizable, $J$ is nilpotent, and $SJ=JS$. The criterion uses selected output rows multiplied by powers of $J$. The following corollary shows how this known result follows from Theorem~\ref{thm:real} and Corollary~\ref{cor:apower}.

\begin{cor}[Nilpotent-component criterion; {\citealp[Lemma~1]{WBT22}}]\label{cor:wbt22}
Let $A\in\mathbb R^{n\times n}$ be constant with real spectrum, and let $C:[0,\infty)\to\mathbb R^{p\times n}$ be measurable and essentially bounded. Write the real Jordan--Chevalley decomposition $A=S+J$, with $S$ diagonalizable, $J$ nilpotent, and $SJ=JS$, and let $q\geq1$ be an integer such that $J^q=0$. If a matrix formed from a fixed selection of rows of $\operatorname{col}(C,CJ,\ldots,CJ^{q-1})$ is PE, then $(A,C(t))$ is UO.
\end{cor}
\begin{proof}
Use the generalized eigenspace decomposition $\mathbb R^n=\bigoplus_\lambda\mathcal G_\lambda$, where $\mathcal G_\lambda:=\ker((A-\lambda I_n)^n)$, and write $x=\sum_\lambda x_\lambda$ with $x_\lambda\in\mathcal G_\lambda$. On $\mathcal G_\lambda$, the matrices $S$ and $J$ act as $\lambda I$ and $A-\lambda I$, respectively. Since these subspaces form a direct sum, $Jx=0$ if and only if $(A-\lambda I_n)x_\lambda=0$ for every $\lambda$. Hence
\[
 \ker J=\bigoplus_{\lambda\in\operatorname{spec}(A)}E_\lambda.
\]

Apply Theorem~\ref{thm:real} and Corollary~\ref{cor:apower} to the auxiliary pair $(J,C(t))$. Because $J$ is nilpotent, the theorem identifies UO of this pair with PE of $C$ on $\ker J$, while the corollary identifies it with PE of $\operatorname{col}(C,CJ,\ldots,CJ^{n-1})$. The stacks through powers $q-1$ and $n-1$ differ only by zero blocks, since $J^q=0$ and every nilpotent $n\times n$ matrix satisfies $J^n=0$. Thus PE of the full stack in the statement is equivalent to PE of $C$ on $\ker J$.

The omitted rows contribute a positive semidefinite term to the averaged quadratic form, so PE of the selected stack implies PE of the full stack. It therefore gives PE of $C$ on $\ker J$, and hence on each $E_\lambda$ separately. Theorem~\ref{thm:real}, now applied to $(A,C(t))$, proves the claim.
\end{proof}

Fixed positive weights on the stack rows affect only the lower constants. A selected PE substack is sufficient but need not be equivalent to the full stack. Testing $C$ on $\ker J$ involves a matrix of dimension $\dim\ker J$ instead of $n$, although the power stack can be easier to recognize in a particular model.

PE on each eigenspace does not imply PE on their sum, because measured vectors from distinct eigenspaces may cancel. The distinct real exponential factors nevertheless allow the dynamics to separate their contributions, as Theorem~\ref{thm:real} establishes. In Example~\ref{ex:plateaus}, $J=0$ and $\ker J=\mathbb R^2$, so both the kernel and nilpotent-power conditions fail, whereas the separate-eigenspace condition holds.

\section{Constant-output cascades}

The preceding criteria concern constant state matrices. A cascade with a time-varying coupling generally lies outside that class, even when the dynamics of its two subsystems are constant. Consider
\begin{equation}
 \begin{aligned}
 \dot\xi&=A_o\xi+G(t)z, & \dot z&=Bz,\\
 q&=C_o\xi, & y&=\operatorname{col}(q,Dz).
 \end{aligned}
 \label{eq:general-cascade}
\end{equation}
Here $\xi\in\mathbb R^m$, $z\in\mathbb R^n$, and $q\in\mathbb R^p$, with $m,n,p\geq1$. The matrices $A_o\in\mathbb R^{m\times m}$, $C_o\in\mathbb R^{p\times m}$, $B\in\mathbb R^{n\times n}$, and $D\in\mathbb R^{d\times n}$ are constant, whereas $G:[0,\infty)\to\mathbb R^{m\times n}$ is time varying. The case $d=0$ corresponds to the absence of direct measurements of $z$. We assume that $(A_o,C_o)$ is observable. Thus $q$ determines $\xi$ when $z=0$; we seek a condition on $G(t)$ and $D$ that characterizes UO of the full state $(\xi,z)$.

\subsection{Reduction to the hidden state}

Let $\nu$ be the observability index of $(A_o,C_o)$, namely the smallest positive integer for which $\mathcal O_-$ below has full column rank $m$:
\begin{equation}
 \begin{aligned}
 \mathcal O_-&:=\operatorname{col}(C_o,C_oA_o,\ldots,C_oA_o^{\nu-1}),\\
 \mathcal O_+&:=\operatorname{col}(C_o,C_oA_o,\ldots,C_oA_o^\nu).
 \end{aligned}
 \label{eq:cascade-stacks}
\end{equation}
Observability gives $1\leq\nu\leq m$. To see how the coupling enters the output derivatives, first differentiate $q$:
\[
 \dot q(t)=C_oA_o\xi(t)+C_oG(t)z(t).
\]
More generally, if $q^{(k)}(t)=C_oA_o^k\xi(t)+R_k(t)z(t)$, then differentiation gives
\[
 \begin{aligned}
 q^{(k+1)}(t)&=C_oA_o^{k+1}\xi(t)\\
 &\quad+\bigl(C_oA_o^kG(t)+\dot R_k(t)+R_k(t)B\bigr)z(t).
 \end{aligned}
\]
This calculation motivates the recursion
\begin{equation}
 \begin{aligned}
 R_0(t)&:=0,\\
 R_{k+1}(t)&:=\dot R_k(t)+R_k(t)B+C_oA_o^kG(t),\\
 &\hspace{38mm}0\leq k<\nu,
 \end{aligned}
 \label{eq:general-recursion}
\end{equation}
where each $R_k(t)\in\mathbb R^{p\times n}$. Whenever the indicated derivatives exist, induction therefore yields
\begin{equation}
 q^{(k)}(t)=C_oA_o^k\xi(t)+R_k(t)z(t),
 \qquad 0\leq k\leq\nu.
 \label{eq:cascade-derivatives}
\end{equation}

To cancel the terms involving $\xi(t)$, choose a constant matrix $N\in\mathbb R^{\ell\times p(\nu+1)}$, where $\ell:=p(\nu+1)-m$, whose rows form a basis of the entire left nullspace of $\mathcal O_+$. Thus
\begin{equation}
 N\mathcal O_+=0,\qquad
 \operatorname{rank}N=\ell,\qquad
 \ker N=\operatorname{im}\mathcal O_+.
 \label{eq:cascade-nullspace}
\end{equation}
Define
\begin{equation}
 \begin{aligned}
 \mathcal R_-(t)&:=\operatorname{col}(R_0(t),\ldots,R_{\nu-1}(t)),\\
 \mathcal R_+(t)&:=\operatorname{col}(R_0(t),\ldots,R_\nu(t)),\\
 \Gamma(t)&:=N\mathcal R_+(t),\qquad
 F(t):=\operatorname{col}(D,\Gamma(t)).
 \end{aligned}
 \label{eq:cascade-reduced-output}
\end{equation}
Writing $Q_+(t):=\operatorname{col}(q(t),\dot q(t),\ldots,q^{(\nu)}(t))$, equation~\eqref{eq:cascade-derivatives} gives
\begin{equation}
 Q_+(t)=\mathcal O_+\xi(t)+\mathcal R_+(t)z(t),
 \qquad NQ_+(t)=\Gamma(t)z(t).
 \label{eq:cascade-elimination}
\end{equation}
Consequently, $F(t)z(t)$ consists of the direct output $Dz(t)$ and combinations of output derivatives in which $\xi(t)$ cancels. The reduced system is
\[
 \dot z=Bz,\qquad \widetilde y=F(t)z.
\]
For this algebraic construction to preserve UO, vanishing integrals of $\|q\|^2$ on a fixed window must force the corresponding derivatives to vanish as well. The uniform regularity assumptions below ensure this implication.

\begin{figure}[tb]
\centering
\begin{tikzpicture}[font=\small,>=stealth,
 dyn/.style={draw,rounded corners=1pt,align=center,minimum height=9mm,inner sep=4pt},
 map/.style={draw,align=center,minimum width=10mm,minimum height=8mm,inner sep=3pt}]
 \node[anchor=west,font=\small\bfseries] at (-.45,.95) {(a) Original cascade};
 \node[dyn] (z) at (.2,0) {$\dot z=Bz$};
 \node[map] (g) at (2.5,0) {$G(t)$};
 \node[dyn] (xi) at (5.1,0) {$\dot\xi=A_o\xi+G(t)z$};
 \draw[->] (z) -- node[above] {$z$} (g);
 \draw[->] (g) -- (xi);
 \node[map] (d) at (.2,-1.3) {$D$};
 \node[map] (c) at (5.1,-1.3) {$C_o$};
 \draw[->] (z) -- (d);
 \draw[->] (xi) -- node[right] {$\xi$} (c);
 \draw[->] (d) -- (.2,-2.05) node[below] {$Dz$};
 \draw[->] (c) -- (5.1,-2.05) node[below] {$q$};
 \node[align=center] at (2.6,-2.8)
   {$\Updownarrow$\quad Equivalent UO (Theorem~\ref{thm:cascade})};
 \node[anchor=west,font=\small\bfseries] at (-.45,-3.4) {(b) Reduced system};
 \node[dyn] (zr) at (.2,-4.25) {$\dot z=Bz$};
 \node[map] (f) at (2.5,-4.25) {$F(t)$};
 \node[inner sep=0pt] (yr) at (5.1,-4.25) {$\widetilde y=F(t)z$};
\draw[->] (zr) -- node[above] {$z$} (f);
 \draw[->,shorten >=3pt] (f) -- (yr);
 \node at (2.6,-5.1) {$F(t)=\operatorname{col}(D,\Gamma(t))$};
\end{tikzpicture}
\caption{The original cascade and the reduced hidden-state system. The derivative relation $NQ_+(t)=\Gamma(t)z(t)$ eliminates $\xi$. Under the regularity assumptions of Theorem~\ref{thm:cascade}, the two systems are UO simultaneously. No spectral restriction on $B$ is needed for this equivalence.}
\label{fig:cascade-criteria}
\end{figure}
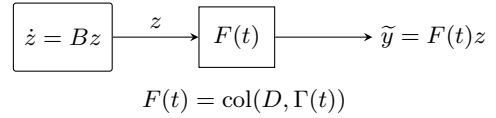

\begin{thm}[Multiple-output cascade reduction]\label{thm:cascade}
Consider~\eqref{eq:general-cascade} with constant matrices $A_o,C_o,B,D$ and observable pair $(A_o,C_o)$. Let $\nu$, $R_k(t)$, $N$, and $F(t)$ be defined by \eqref{eq:cascade-stacks}, \eqref{eq:general-recursion}, \eqref{eq:cascade-nullspace}, and~\eqref{eq:cascade-reduced-output}. Assume that $G,R_0,\ldots,R_\nu$ are continuous and uniformly bounded on $[0,\infty)$, that all derivatives in \eqref{eq:general-recursion} exist, and that $R_\nu$ is uniformly continuous on $[0,\infty)$. Then the full system is UO if and only if the reduced pair $(B,F(t))$ is UO.
\end{thm}

The proof is given in Appendix~\ref{app:proof-cascade}.

A convenient sufficient condition for the theorem's regularity assumptions is $G\in\mathcal C_b^\nu$. Indeed, each $R_k(t)$ is a finite sum of constant matrix multiples of derivatives of $G(t)$ through order $k-1$, and then $\dot R_\nu$ is bounded. The stated assumptions also allow weaker regularity, provided that $R_\nu$ is uniformly continuous. In particular, the compactness step in Appendix~\ref{app:proof-cascade} controls the highest derivative on the whole closed observation window; boundedness alone would not suffice, as illustrated in Remark~\ref{rem:chirp}.

The reduced UO condition does not depend on the chosen basis of the left nullspace. Any other basis has the form $\widetilde N=UN$ with $U$ constant and invertible, so $\widetilde\Gamma(t)=U\Gamma(t)$ and the squared norms of $\operatorname{col}(Dz,\widetilde\Gamma(t)z)$ and $F(t)z$ bound each other by fixed positive factors. All matrices in the criterion are formed from the known coefficients and their derivatives; its evaluation does not require differentiating measured data.

Linear combinations of output derivatives that eliminate a state also occur in continuous-time parity relations \citep[Section~III, equations~(17)--(19)]{Varga09}. Theorem~\ref{thm:cascade} proves equivalence of the full and reduced UO conditions for~\eqref{eq:general-cascade}. The implication from small measured outputs to small derivative outputs follows the strategy of \citet{MES17} and \citet{HHS20}, with the compactness argument in Appendix~\ref{app:proof-cascade} accommodating uniform continuity of $R_\nu$. Neither implication imposes a spectral or stability condition on $A_o$ or $B$.

\begin{rem}[Increasingly rapid oscillations]\label{rem:chirp}
The uniform regularity assumption cannot be replaced by boundedness and smoothness alone. Consider $\dot h=v$, $\dot v=b(t)z$, $\dot z=0$, and $y=h$, with scalar $z$ and $b(t):=\sin(t^2)$. Here $\Gamma(t)=b(t)$, so the reduced pair $(0,b(t))$ is UO precisely when $b$ is PE. To verify PE, write
\[
 \int_t^{t+1}\sin^2(s^2)ds
 =\frac12-\frac12\int_t^{t+1}\cos(2s^2)ds.
\]
Integration by parts gives, for $t>0$,
\[
 \int_t^{t+1}\cos(2s^2)ds
 =\left[\frac{\sin(2s^2)}{4s}\right]_t^{t+1}
 +\int_t^{t+1}\frac{\sin(2s^2)}{4s^2}ds.
\]
Its absolute value is at most $1/(2t)+1/(4t^2)$, so the integral of $\sin^2(s^2)$ is at least $1/4$ for all sufficiently large $t$. On the remaining compact interval of initial times, that integral is continuous and strictly positive. Thus $b$ is PE with a common unit window.

Nevertheless, fix any $\delta>0$ and choose the unit initial state $(h(t),v(t),z(t))=(0,0,1)$. For $0\leq u\leq\delta$, integration by parts gives
\[
 \begin{aligned}
 v(t+u)
 &=\int_0^u\sin((t+s)^2)ds\\
 &=-\left[\frac{\cos((t+s)^2)}{2(t+s)}\right]_{s=0}^{s=u}\\
 &\quad-\int_0^u\frac{\cos((t+s)^2)}{2(t+s)^2}ds.
 \end{aligned}
\]
Consequently, $|v(t+u)|\leq1/t+u/(2t^2)$ and $|h(t+u)|\leq u/t+u^2/(4t^2)$. Therefore
\[
 \int_0^\delta|h(t+s)|^2ds
 \leq\delta\left(\frac{\delta}{t}+\frac{\delta^2}{4t^2}\right)^2
 \longrightarrow0.
\]
No fixed observation window admits a uniform positive Gramian bound. Here $b$ is smooth and bounded but not uniformly continuous: the times $s_k=\sqrt{2\pi k+\pi/2}$ and $u_k=\sqrt{2\pi k+3\pi/2}$ satisfy $u_k-s_k\to0$, whereas $b(s_k)-b(u_k)=2$. The coupling can therefore be PE while its effect after integration becomes arbitrarily small.
\end{rem}

\subsection{Excitation on unmeasured eigenspaces}

The reduction separates the role of the measured subsystem from the observability of $z$. When $B$ has real spectrum, Theorem~\ref{thm:real} characterizes the reduced pair through $F(t)=\operatorname{col}(D,\Gamma(t))$ on each eigenspace of $B$. On $E_\lambda:=\ker(B-\lambda I_n)$, the test is a lower bound for $\|Dv\|^2+\delta^{-1}\int_t^{t+\delta}\|\Gamma(s)v\|^2ds$. The term $\|Dv\|^2$ vanishes precisely on $E_\lambda\cap\ker D$. The next corollary proves that excitation need only be checked on this intersection.

\begin{cor}[Excitation on unmeasured eigenspaces]\label{cor:hidden}
Assume the hypotheses and definitions of Theorem~\ref{thm:cascade}, and suppose $\operatorname{spec}(B)\subset\mathbb R$. Then the full cascade~\eqref{eq:general-cascade} is UO if and only if $\Gamma$ is PE on each subspace
\[
 U_\lambda:=\ker(B-\lambda I_n)\cap\ker D,
 \qquad \lambda\in\operatorname{spec}(B),
\]
separately.
\end{cor}
\begin{proof}
By Theorems~\ref{thm:cascade} and~\ref{thm:real}, UO of the full system is equivalent to PE of $F(t)=\operatorname{col}(D,\Gamma(t))$ on each eigenspace $E_\lambda=\ker(B-\lambda I_n)$. Since $D$ vanishes on $U_\lambda$, this proves necessity.

Conversely, suppose $\Gamma$ is PE on each $U_\lambda$. Choose common PE constants $\delta,\beta>0$, enlarging the finitely many windows and decreasing the bounds if necessary, and set
\[
 Q_t:=\frac1\delta\int_t^{t+\delta}\Gamma(s)^\top\Gamma(s)ds.
\]
The matrices $Q_t$ are positive semidefinite and uniformly bounded. If $D^\top D+Q_t$ had no uniform positive lower bound on some $E_\lambda$, there would be times $t_j$ and unit vectors $w_j\in E_\lambda$ for which $\|Dw_j\|^2+w_j^\top Q_{t_j}w_j\to0$. Passing to a subsequence gives $w_j\to w\in E_\lambda$ with $\|w\|=1$ and $Dw=0$. Uniform boundedness of $Q_t$ also gives $w^\top Q_{t_j}w\to0$. But $w\in U_\lambda$, so the PE hypothesis requires $w^\top Q_{t_j}w\geq\beta$, a contradiction. Taking the minimum of the lower bounds over the finitely many eigenspaces proves PE of $F$ on each of them.
\end{proof}

A zero subspace $U_\lambda$ imposes no condition. In particular, if every $U_\lambda$ is zero, the direct measurements suffice for UO, regardless of the coupling within the stated regularity class.

\subsection{Two illustrations of the reduction}

\begin{exmp}[Two measured channels]\label{ex:multiple-channels}
This example has two measured channels of different derivative orders and an observable measured subsystem that admits no observable scalar output. Let
\[
 A_o=\begin{bmatrix}0&1&0\\0&0&0\\0&0&0\end{bmatrix},
 \qquad C_o=\begin{bmatrix}1&0&0\\0&0&1\end{bmatrix},
\]
\[
 G(t)=\begin{bmatrix}0&0\\\sin t&0\\0&\cos t\end{bmatrix},
\]
with $B=0_{2\times2}$ and no direct output. Thus $q_1=\xi_1$ and $q_2=\xi_3$. The pair $(A_o,C_o)$ has $\nu=2$, and
\[
 R_1(t)=\begin{bmatrix}0&0\\0&\cos t\end{bmatrix},
 \qquad R_2(t)=\begin{bmatrix}\sin t&0\\0&-\sin t\end{bmatrix}.
\]
The first three rows of $\mathcal O_+$ are linearly independent, and its last three rows vanish. Selecting those last three rows gives a basis $N$ and hence
\[
 \Gamma(t)=\begin{bmatrix}0&\cos t\\\sin t&0\\0&-\sin t\end{bmatrix},
\]
\[
 \frac1\pi\int_t^{t+\pi}\Gamma(s)^\top\Gamma(s)ds
 =\begin{bmatrix}1/2&0\\0&1\end{bmatrix}.
\]
The cascade is UO by Corollary~\ref{cor:hidden}. Both measured channels are needed for the hypothesis on $(A_o,C_o)$: since $\dim\ker A_o=2$, every scalar output row vanishes on some nonzero vector in $\ker A_o$.
\end{exmp}

\begin{exmp}[A second-order cascade with direct measurements]\label{ex:direct-cascade}
Consider
\begin{equation}
 \dot h=v,\quad \dot v=b(t)^\top z,\quad \dot z=0,
 \label{eq:second-order}
\end{equation}
with output $y=\operatorname{col}(h,Dz)$, where $D$ is constant and $b$ is bounded and uniformly continuous. Here $m=2$, $B=0$, $R_1(t)=0$, and $R_2(t)=b(t)^\top$. Corollary~\ref{cor:hidden} therefore gives UO if and only if $b^\top$ is PE on $\ker D$.

For the ideal LTV model in \citet[equations~(37)--(38)]{TBH26},
\[
 b(t)^\top=-e_3^\top[a_I(t)]_\times,\qquad D=-[m_I]_\times.
\]
The notation $[u]_\times v=u\times v$ denotes the cross-product matrix, $a_I$ is the acceleration in inertial coordinates, $m_I\ne0$ is a constant reference magnetic vector, and $e_3$ is the vertical direction. Assume that $a_I$ is bounded and uniformly continuous; bounded $a_I,\dot a_I$ is a sufficient condition. Since $\ker D=\operatorname{span}(m_I)$, writing $\widehat m_I:=m_I/\|m_I\|$ gives the exact condition: there exist $\delta,\beta>0$ such that
\begin{equation}
 \frac1\delta\int_t^{t+\delta}
 \bigl|e_3^\top[a_I(s)]_\times\widehat m_I\bigr|^2ds\geq\beta
 \label{eq:scalar-motion}
\end{equation}
for every $t\geq0$. The scalar in the integrand is $a_I^\top(\widehat m_I\times e_3)$. It vanishes if $m_I$ is parallel to $e_3$.

To compare with \citet[Lemma~2]{TBH26}, let $J_h:=[e_1\ e_2]$ and $a_\perp(t):=e_3^\top[a_I(t)]_\times J_h$. Their sufficient condition requires PE of $a_\perp$ on $\mathbb R^2$. Writing $w:=J_h^\top\widehat m_I$, one has $e_3^\top[a_I]_\times\widehat m_I=a_\perp w$, so that condition implies \eqref{eq:scalar-motion} whenever $m_I$ is not parallel to $e_3$. The converse fails: for $m_I=e_1$ and $a_I(t)=\sin(t)e_2$, the left side of \eqref{eq:scalar-motion} equals $1/2$ on every window of length $\pi$, whereas every averaged matrix $a_\perp^\top a_\perp$ has rank one. The comparison assumes the uniform regularity imposed here and concerns the LTV model; nonlinear observer stability requires its separate analysis.
\end{exmp}

\section{Case study: integrator chains with range measurements}\label{sec:range-case-study}

This application combines the two main results. We first lift the quadratic measurements into a cascade and apply Theorem~\ref{thm:cascade} to eliminate its measured subsystem. The reduced dynamics are constant and nilpotent, so Theorem~\ref{thm:real} then gives an exact excitation condition on their eigenspace.

\emph{Model and finite lift.}
Consider an arbitrary-order position model with known input and fixed anchors,
\begin{equation}
 p^{(r)}=u(t),\qquad y_i=\tfrac12\|p-a_i\|^2,
 \quad i=1,\ldots,L,
 \label{eq:range-chain}
\end{equation}
where $p,u,a_i\in\mathbb R^d$, $r\geq1$, and $L\geq1$. The outputs are obtained by squaring the measured ranges and dividing by two. Assume $u\in\mathcal C_b^{r-1}$ and that $u^{(r-1)}$ is uniformly continuous; $u\in\mathcal C_b^r$ is sufficient. The state extensions in \citet{TBD21,BTHD23} motivate the following construction, for which the exact excitation condition can be evaluated explicitly.

Set $z_0:=p-a_1$, $z_j:=p^{(j)}$ for $1\leq j<r$, and $z:=\operatorname{col}(z_0,\ldots,z_{r-1})$. Let $S_s$ denote the $s\times s$ matrix with ones on its first superdiagonal and zeros elsewhere, and let $E_j\in\mathbb R^{d\times rd}$ select the block $z_j$. Then
\[
 B:=S_r\otimes I_d,\qquad \dot z=Bz+E_{r-1}^\top u(t).
\]
Here $\otimes$ denotes the Kronecker product. Introduce $m:=2r-1$ auxiliary states
\begin{equation}
 \begin{aligned}
 Q_0&:=E_0^\top E_0,& Q_{k+1}&:=B^\top Q_k+Q_kB,\\
 \xi_k&:=\tfrac12z^\top Q_kz,& \xi&:=\operatorname{col}(\xi_0,\ldots,\xi_{m-1}).
 \end{aligned}
 \label{eq:range-lift}
\end{equation}
Here $0\leq k<m$ in the recursions. The nonzero $(i,j)$ blocks of $Q_k$, indexed from zero, are $\binom{k}{i}I_d$ when $i+j=k$; hence $Q_m=0$. Differentiating the quadratic forms, including the known-input contribution, gives the exact finite-dimensional model
\begin{equation}
 \begin{aligned}
 \dot\xi&=S_m\xi+G(t)z, & y_1&=e_1^\top\xi,\\
 e_{k+1}^\top G(t)&:=u(t)^\top E_{r-1}Q_k,
 &&0\leq k<m.
 \end{aligned}
 \label{eq:range-lifted-dynamics}
\end{equation}
The matrices $Q_k$ are obtained by differentiating the quadratic output along the zero-input dynamics. The term $G(t)z$ accounts for the known input in the actual dynamics of the auxiliary states.

The other ranges provide constant linear outputs through their differences:
\[
 \tfrac12\|a_i-a_1\|^2+y_1-y_i=(a_i-a_1)^\top z_0,
 \qquad i=2,\ldots,L.
\]
Define $H:=\operatorname{col}((a_2-a_1)^\top,\ldots,(a_L-a_1)^\top)$ and $D:=HE_0$, taking $H$ to have no rows when $L=1$. After subtracting known offsets, the range outputs are related to $\operatorname{col}(e_1^\top\xi,Dz)$ by an invertible constant transformation. Differences of two lifted trajectories driven by the same input therefore satisfy \eqref{eq:general-cascade} with $A_o=S_m$ and $C_o=e_1^\top$. The extension adds $2r-1$ scalar states, independently of the number of anchors.

\emph{Cascade reduction.}
For the measured chain, $\nu=m$, $\mathcal O_-=I_m$, and $\mathcal O_+=\operatorname{col}(I_m,0)$. Choosing $N=[0_{1\times m}\ \ 1]$ gives $\Gamma(t)=R_m(t)$, with
\begin{equation}
 \begin{aligned}
 R_0(t)&:=0,\\
 R_{k+1}(t)&:=\dot R_k(t)+R_k(t)B+e_{k+1}^\top G(t),\\
 &\hspace{38mm}0\leq k<m.
 \end{aligned}
 \label{eq:recursion}
\end{equation}
The first $r-1$ rows of $G(t)$ vanish. For $r-1\leq k<m$, its row $e_{k+1}^\top G(t)$ has only block $k-r+1$ nonzero, equal to $\binom{k}{r-1}u(t)^\top$. Thus $R_k(t)=0$ for $k<r$, and $R_r(t)=u(t)^\top E_0$. More generally, induction in \eqref{eq:recursion} gives
\[
 R_k(t)E_j^\top=\binom{k}{j}u^{(k-r-j)}(t)^\top,
 \qquad 0\leq j\leq k-r,
\]
for $r\leq k\leq m$; the blocks with $j>k-r$ vanish. In the induction step, differentiation preserves the block index, multiplication by $B$ shifts it by one, and the coefficients combine by Pascal's identity, including the boundary contribution from $G(t)$. At $k=m$, this yields
\begin{equation}
 R_m(t)E_j^\top=\binom{2r-1}{j}u^{(r-1-j)}(t)^\top,
 \qquad 0\leq j<r.
 \label{eq:range-reduced-row}
\end{equation}
The recursion gives $R_k=0$ for $k<r$, and $R_k$ involves derivatives of $u$ only through order $k-r$ for $r\leq k\leq m$. Thus $R_0,\ldots,R_{m-1}$ are continuously differentiable and bounded, while $R_m$ is bounded and uniformly continuous. Theorem~\ref{thm:cascade} therefore proves that the lifted system is UO if and only if the reduced pair $(B,\operatorname{col}(D,R_m(t)))$ is UO.

\emph{Eigenspace test.}
The matrix $B$ is nilpotent, with $\ker B=\operatorname{im}E_0^\top$. Theorem~\ref{thm:real} therefore makes reduced UO equivalent to PE of $\operatorname{col}(H,u^{(r-1)}(t)^\top)$ on $\mathbb R^d$. Since $H$ is constant, the restriction argument in Corollary~\ref{cor:hidden} removes the directions already measured by $H$. More explicitly,
\[
 \begin{aligned}
 \ker B\cap\ker D&=\{E_0^\top w:w\in\ker H\},\\
 R_m(t)E_0^\top&=u^{(r-1)}(t)^\top.
 \end{aligned}
\]
These reductions give the following characterization.

\begin{prop}[Uniform observability from range measurements]\label{prop:range-uo}
Consider model~\eqref{eq:range-chain} with fixed anchors and the lifted LTV system constructed above. Assume $u\in\mathcal C_b^{r-1}$ and that $u^{(r-1)}$ is uniformly continuous. Let $H$ be the anchor-difference matrix defined above. The homogeneous error system of this lift, with output $\operatorname{col}(e_1^\top\xi,Dz)$, is UO on its full state space $\mathbb R^{rd+2r-1}$ if and only if $u^{(r-1)}(t)^\top$ is PE on $\ker H$, that is, there exist $\delta,\mu>0$ such that
\begin{equation}
 \begin{gathered}
 \frac1\delta\int_t^{t+\delta}
 |u^{(r-1)}(s)^\top w|^2\,ds\geq\mu\|w\|^2,\\
 t\geq0,\qquad w\in\ker H.
 \end{gathered}
 \label{eq:range-pe}
\end{equation}
\end{prop}
\begin{proof}
The preceding cascade reduction gives the equivalent hidden pair $(B,\operatorname{col}(D,R_m(t)))$. Since $\ker B\cap\ker D=\{E_0^\top w:w\in\ker H\}$ and $R_m(t)E_0^\top=u^{(r-1)}(t)^\top$, Corollary~\ref{cor:hidden} gives exactly~\eqref{eq:range-pe}.
\end{proof}

\emph{Geometric interpretation.}
The anchor differences span the directions measured by the range differences, while
\[
 \ker H=\operatorname{span}(a_2-a_1,\ldots,a_L-a_1)^\perp.
\]
Proposition~\ref{prop:range-uo} therefore requires excitation only in the orthogonal directions. The condition is vacuous when the anchors affinely span $\mathbb R^d$, and reduces to full-space PE of $u^{(r-1)}(t)^\top$ for a single anchor. For example, in three dimensions with $\operatorname{rank}H=2$, it requires a uniform positive lower bound on the mean square of $u^{(r-1)}$ along a unit normal to the anchor plane.

Multiple-range observability conditions combining anchor geometry with motion excitation were already established by \citet[Lemmas~4.1--4.2]{HS17}; see \citet[Section~3]{HHS20} for the corrected observability argument. In the known-velocity case ($r=1$), their sum of a constant anchor term and an averaged velocity term is equivalent, under the boundedness assumed here, to~\eqref{eq:range-pe}: their anchor matrix and $H$ have the same kernel, so the restriction argument in the proof of Corollary~\ref{cor:hidden} applies. The subspace formulation thus recovers the geometry of that existing condition. The present case study illustrates how the two general theorems yield this separation and provide a necessary and sufficient UO criterion for the specified arbitrary-order lifted LTV system under the stated regularity assumptions.

\emph{Comparison with full-space PE.}
For $r=2$, the extension is $\xi=\operatorname{col}(\tfrac12\|z_0\|^2,z_0^\top z_1,\|z_1\|^2)$ and $R_3(t)=[\dot u(t)^\top,3u(t)^\top]$, recovering the model in \citet[Section~6]{TBD21}. With one anchor, $d=2$, and $u(t)=\operatorname{col}(\cos t,\sin t)$, condition~\eqref{eq:range-pe} holds with $\delta=\pi$ and $\mu=1/2$. Yet
\[
 R_3(t)=\cos t\,[0,1,3,0]+\sin t\,[-1,0,0,3],
\]
so every averaged matrix $R_3^\top R_3$ has rank at most two. Full-space PE of $R_3$, as required by the sufficient condition in \citet[Proposition~7]{TBD21}, fails in dimension four. Proposition~6 of that reference gives sufficiency of reduced UO; Theorem~\ref{thm:cascade} supplies the converse under the regularity imposed here. The equivalence \eqref{eq:range-pe} concerns the lifted LTV system on its full state space. It does not assert necessity for observability of the original nonlinear model, whose lifted states satisfy the constraints in \eqref{eq:range-lift}.

\section{Conclusion}
For constant real-spectrum dynamics, uniform observability is equivalent to persistent excitation on each eigenspace separately. This characterization holds for measurable, essentially bounded output matrices and admits an equivalent averaged Kalman test. It requires neither output differentiation nor instantaneous full rank, and clarifies why excitation on larger subspaces can impose unnecessary restrictions.

For cascades with an observable constant measured subsystem, elimination through a finite derivative stack gives an exact observability test for the hidden dynamics under explicit uniform regularity conditions. The reduction itself requires no spectral restriction. When the hidden spectrum is real, excitation is needed only on each eigenspace intersected with the kernel of the direct measurement matrix. The rapid-oscillation example shows that boundedness and smoothness alone do not suffice for this reduction.

For an arbitrary-order position chain measured through ranges to fixed anchors, the two reductions yield PE of a known input derivative only on the orthogonal complement of the anchor differences. This condition characterizes UO of the full lifted LTV model.

A remaining question is how to formulate similarly direct tests when the state matrix has nonreal eigenvalues, including purely imaginary pairs. The rotation example shows that such tests must account for cancellations between oscillatory dynamics and the time-varying output.

\begin{ack}
The author thanks Professors Tarek Hamel and Abdelhamid Tayebi for valuable discussions during the development of many of the results presented in this paper.

This work was supported in part by the Natural Sciences and Engineering Research Council of Canada (NSERC) under Discovery Grant RGPIN-2020-04759 and by the Fonds de recherche du Qu\'ebec (FRQ).
\end{ack}

\appendix

\section{Proof of Theorem~\ref{thm:real}}\label{app:proof-real}

\begin{proof}
\emph{Necessity.} Suppose that $W_{\delta_o}(t)\succeq\mu_o I_n$ for every $t\geq0$. For $v\in E_\lambda$, put $\lambda_+:=\max\{\lambda,0\}$. Then
\[
 \begin{aligned}
 \mu_o\|v\|^2
 &\leq\frac1{\delta_o}\int_0^{\delta_o}e^{2\lambda s}\|C(t+s)v\|^2ds\\
 &\leq\frac{e^{2\lambda_+\delta_o}}{\delta_o}
       \int_t^{t+\delta_o}\|C(s)v\|^2ds.
 \end{aligned}
\]
This proves PE on each eigenspace.

\emph{Sufficiency.} For each eigenvalue, choose PE constants $\delta_\lambda,\mu_\lambda>0$. Enlarging the integration windows gives common constants $\delta=\max_\lambda\delta_\lambda$ and $\mu=\min_\lambda(\mu_\lambda\delta_\lambda)/\delta>0$. Let $M:=\operatorname*{ess\,sup}_t\|C(t)\|$; PE implies $M\geq\sqrt\mu>0$. We first obtain a bound uniform in the initial time for each fixed unit initial state, and then use compactness to make the bound uniform over all such states.

The generalized eigenspaces $\mathcal G_\lambda:=\ker((A-\lambda I_n)^n)$ give the direct sum $\mathbb R^n=\bigoplus_\lambda\mathcal G_\lambda$. Write a unit vector $x$ uniquely as $x=\sum_\lambda x_\lambda$, with $x_\lambda\in\mathcal G_\lambda$. Choose $\lambda_x:=\max\{\lambda:x_\lambda\ne0\}$, and let $k_x$ be the largest nonnegative integer for which $(A-\lambda_x I_n)^{k_x}x_{\lambda_x}\ne0$. Then $0\leq k_x\leq n-1$, and $v_x:=(A-\lambda_x I_n)^{k_x}x_{\lambda_x}$ is a nonzero vector in $E_{\lambda_x}$. The finite expansion
\[
 e^{As}x=\sum_\lambda e^{\lambda s}
          \sum_{j=0}^{n-1}\frac{s^j}{j!}(A-\lambda I_n)^jx_\lambda
\]
therefore has the form
\begin{equation}
 e^{As}x=a_x(s)\bigl(v_x+\eta_x(s)\bigr),\qquad s>0,
 \label{eq:dominant}
\end{equation}
where $a_x(s):=e^{\lambda_xs}s^{k_x}/k_x!$ and $\eta_x(s)\to0$ as $s\to\infty$. To see the latter, divide the expansion by $a_x(s)$. Terms belonging to $\lambda_x$ have lower polynomial degree than $k_x$, except for $v_x$ itself. Terms belonging to smaller eigenvalues contain an exponentially decaying factor multiplying a fixed power of $s$.

Choose $r_x\geq1$ such that $\|\eta_x(s)\|\leq(\sqrt\mu/(2M))\|v_x\|$ for $s\geq r_x$, and set $I_x:=[r_x,r_x+\delta]$. PE and boundedness of $C$ give, for every $t\geq0$,
\[
 \begin{aligned}
 \int_{I_x}\|C(t+s)v_x\|^2ds&\geq\mu\delta\|v_x\|^2,\\
 \int_{I_x}\|C(t+s)\eta_x(s)\|^2ds&\leq\frac{\mu\delta}{4}\|v_x\|^2.
 \end{aligned}
\]
The triangle inequality in $L^2(I_x)$ consequently gives
\[
 \int_{I_x}\|C(t+s)(v_x+\eta_x(s))\|^2ds
 \geq\frac{\mu\delta}{4}\|v_x\|^2.
\]
Since $a_x$ is continuous and positive on $I_x$, its squared minimum $c_x:=\min_{s\in I_x}a_x(s)^2$ is positive. Set $\delta_x:=r_x+\delta$ and $b_x:=c_x\mu\delta\|v_x\|^2/4>0$. Equation~\eqref{eq:dominant} yields
\begin{equation}
 \int_0^{\delta_x}\|C(t+s)e^{As}x\|^2ds\geq b_x,\qquad t\geq0.
 \label{eq:directionbound}
\end{equation}

To obtain a common window, define
\[
 L_x:=M\left(\int_0^{\delta_x}\|e^{As}\|^2ds\right)^{1/2},
 \qquad \rho_x:=\frac{\sqrt{b_x}}{2L_x}.
\]
The $L^2(0,\delta_x)$ norm of $C(t+s)e^{As}(u-x)$ is at most $L_x\|u-x\|$. Applying the triangle inequality once more to \eqref{eq:directionbound} shows that
\[
 \int_0^{\delta_x}\|C(t+s)e^{As}u\|^2ds>\frac{b_x}{4}
\]
whenever $\|u-x\|<\rho_x$, uniformly in $t$. The neighborhoods $\mathcal U_x:=\{u\in\mathbb S^{n-1}:\|u-x\|<\rho_x\}$ cover the unit sphere. Choose a finite subcover associated with $x_1,\ldots,x_N$, and put $\delta_o:=\max_i\delta_{x_i}$ and $\mu_o:=\min_i b_{x_i}/(4\delta_o)>0$. For any unit vector $u$, choose $i$ with $u\in\mathcal U_{x_i}$. Enlarging the integration interval from $[0,\delta_{x_i}]$ to $[0,\delta_o]$ gives
\[
 u^\top W_{\delta_o}(t)u\geq\frac{b_{x_i}}{4\delta_o}\geq\mu_o.
\]
Thus $W_{\delta_o}(t)\succeq\mu_o I_n$ for every $t\geq0$.
\end{proof}

\section{Proof of Corollary~\ref{cor:apower}}\label{app:proof-apower}

\begin{proof}
Suppose first that $(A,C(t))$ is UO. By Theorem~\ref{thm:real}, choose common PE constants $\delta,\mu>0$ for its eigenspaces. Let $V_\lambda$ have orthonormal columns spanning $E_\lambda$, and define
\[
 Q_t:=\frac1\delta\int_t^{t+\delta}C(s)^\top C(s)ds,
 \qquad M:=\operatorname*{ess\,sup}_s\|C(s)\|.
\]
For a symmetric matrix $Q$, define
\[
 \mathcal F_A(Q):=\sum_{k=0}^{n-1}(A^k)^\top QA^k.
\]
Constancy of $A$ gives
\[
 \frac1\delta\int_t^{t+\delta}\mathcal O_A(s)^\top\mathcal O_A(s)ds
 =\mathcal F_A(Q_t).
\]
The symmetric matrices satisfying $0\preceq Q\preceq M^2I_n$ and $V_\lambda^\top QV_\lambda\succeq\mu I_{\dim E_\lambda}$ for every $\lambda$ form a compact set $\mathcal K_A$ containing every $Q_t$.

We claim that $\mathcal F_A(Q)$ is positive definite for every $Q\in\mathcal K_A$. Otherwise, there would be $x\ne0$ with
\[
 0=x^\top\mathcal F_A(Q)x
   =\sum_{k=0}^{n-1}\|Q^{1/2}A^kx\|^2.
\]
Hence $Q^{1/2}$ would vanish on the subspace
\[
 \mathcal X:=\operatorname{span}(x,Ax,\ldots,A^{n-1}x).
\]
By the Cayley--Hamilton theorem, $\mathcal X$ is nonzero and $A$-invariant. The restriction of $A$ to $\mathcal X$ has real spectrum and therefore has a real eigenvector $v\ne0$. This vector belongs to some $E_\lambda$ and satisfies $Q^{1/2}v=0$, contradicting $v^\top Qv\geq\mu\|v\|^2$. Compactness of $\mathcal K_A$ and continuity of $\lambda_{\min}(\mathcal F_A(Q))$ now give a positive uniform lower bound. Thus $\mathcal O_A$ is PE.

Conversely, suppose that $\mathcal O_A$ is PE with constants $\delta,\beta>0$, and define $Q_t$ with this window. For $v\in E_\lambda$,
\[
 \beta\|v\|^2\leq v^\top\mathcal F_A(Q_t)v
 =\left(\sum_{k=0}^{n-1}\lambda^{2k}\right)v^\top Q_tv.
\]
The sum is a positive constant for each eigenvalue, so $C$ is PE on every $E_\lambda$. Theorem~\ref{thm:real} proves UO.
\end{proof}

\section{Proof of Theorem~\ref{thm:cascade}}\label{app:proof-cascade}

\begin{proof}
For $k<\nu$, the recursion gives
\[
 \dot R_k(t)=R_{k+1}(t)-R_k(t)B-C_oA_o^kG(t).
\]
Thus $R_0,\ldots,R_{\nu-1}$ have continuous, uniformly bounded first derivatives, and the identities \eqref{eq:cascade-derivatives}--\eqref{eq:cascade-elimination} hold along every trajectory.

\emph{Sufficiency.}
Suppose that $(B,F(t))$ is UO, and let $\delta,\mu_r>0$ be an observation window and a normalized lower bound for this pair. We proceed by contradiction, following the sequence construction in \citet[Proposition~6]{TBD21}; see also \citet[Appendix~A.1]{Scandaroli13}. Denote by $\Phi_c(s,t)$ the state-transition matrix of the full cascade, and set
\[
 \mathcal C:=\begin{bmatrix}C_o&0\\0&D\end{bmatrix}.
\]
Assume that the full system is not UO. Consider a sequence $(\mu_p)_{p\geq1}$ such that $\mu_p>0$ and $\lim_{p\to\infty}\mu_p=0$. For the fixed window $\delta$, there exist a sequence of times $(t_p)_{p\geq1}\subset[0,\infty)$ and a sequence of vectors $(\widehat v_p)_{p\geq1}$ in $\mathbb S^{m+n-1}$ such that
\[
 \frac1\delta\int_{t_p}^{t_p+\delta}
 \|\mathcal C\Phi_c(s,t_p)\widehat v_p\|^2ds<\mu_p,
 \qquad p\geq1.
\]
Since $\mathbb S^{m+n-1}$ is compact, we can extract a subsequence, still indexed by $p$, for which
\[
 \lim_{p\to\infty}\widehat v_p=v
   =\operatorname{col}(v_\xi,v_z),\qquad \|v\|=1.
\]
The full state matrix is uniformly bounded, so $\Phi_c(s,t_p)$ is bounded uniformly in $p$ for $s\in[t_p,t_p+\delta]$. The same is true of $\mathcal C\Phi_c(s,t_p)$. Since the integration windows have fixed length, we may replace $\widehat v_p$ by $v$ in the preceding vanishing integral and obtain
\begin{equation}
 \lim_{p\to\infty}\int_{t_p}^{t_p+\delta}
 \|\mathcal C\Phi_c(s,t_p)v\|^2ds=0.
 \label{eq:cascade-small-output}
\end{equation}

Let $\xi_p,z_p,q_p$ denote the trajectories starting at time $t_p$ from the fixed initial state $(v_\xi,v_z)$. Thus
\[
 \begin{aligned}
 \operatorname{col}(\xi_p(t),z_p(t))&=\Phi_c(t,t_p)v,\\
 z_p(t)&=e^{B(t-t_p)}v_z,\qquad q_p(t)=C_o\xi_p(t).
 \end{aligned}
\]
Define $f_p(s):=q_p(t_p+s)$ for $s\in[0,\delta]$. By a change of integration variable, \eqref{eq:cascade-small-output} becomes
\[
 \lim_{p\to\infty}\int_0^\delta
 \bigl(\|f_p(s)\|^2+\|De^{Bs}v_z\|^2\bigr)ds=0.
\]
In particular, $\lim_{p\to\infty}\int_0^\delta\|f_p(s)\|^2ds=0$. From \eqref{eq:cascade-derivatives}, the successive derivatives of $f_p$ satisfy
\begin{equation}
 \begin{aligned}
 f_p^{(k)}(s)
 &=C_oA_o^k\xi_p(t_p+s)+R_k(t_p+s)e^{Bs}v_z,\\
 &\hspace{30mm}0\leq k\leq\nu.
 \end{aligned}
 \label{eq:cascade-translated-derivatives}
\end{equation}
The trajectories and their first derivatives are uniformly bounded on the observation windows. Hence $f_p^{(k)}$, $0\leq k\leq\nu$, are uniformly bounded on $[0,\delta]$. For $k<\nu$, boundedness of $f_p^{(k+1)}$ makes the family $(f_p^{(k)})_p$ equicontinuous. For the highest derivative, \eqref{eq:cascade-translated-derivatives} and uniform continuity of $R_\nu$ give
\[
 \|f_p^{(\nu)}(s)-f_p^{(\nu)}(u)\|
 \leq c_\delta|s-u|+M_\delta\omega(|s-u|),
\]
for $s,u\in[0,\delta]$, where $c_\delta,M_\delta$ are independent of $p$, and $\omega$ is a modulus of continuity of $R_\nu$ satisfying $\lim_{h\to0^+}\omega(h)=0$. Thus $(f_p^{(\nu)})_p$ is also equicontinuous.

By Arzel\`a--Ascoli, we can extract a further subsequence, still indexed by $p$, along which $f_p^{(k)}$ converges uniformly to a continuous function $g_k$ on $[0,\delta]$, simultaneously for $k=0,\ldots,\nu$. The vanishing integral of $\|f_p\|^2$ implies $g_0=0$. Passing to the limit in
\[
 f_p^{(k)}(s)=f_p^{(k)}(0)+\int_0^s f_p^{(k+1)}(u)du
\]
gives $g_k(s)=g_k(0)+\int_0^s g_{k+1}(u)du$ for $k<\nu$. It follows successively that $g_1,\ldots,g_\nu$ are zero. Hence
\begin{equation}
 \lim_{p\to\infty}\int_0^\delta\|f_p^{(k)}(s)\|^2ds=0,
 \qquad k=0,\ldots,\nu.
 \label{eq:cascade-small-derivatives}
\end{equation}

Using \eqref{eq:cascade-elimination}, we have
\[
 \Gamma(t_p+s)e^{Bs}v_z
   =N\operatorname{col}(f_p(s),\ldots,f_p^{(\nu)}(s)).
\]
Consequently,
\[
 \begin{aligned}
 &\int_0^\delta\|\Gamma(t_p+s)e^{Bs}v_z\|^2ds\\
 &\qquad\leq\|N\|^2\sum_{k=0}^\nu
       \int_0^\delta\|f_p^{(k)}(s)\|^2ds.
 \end{aligned}
\]
Together with \eqref{eq:cascade-small-output} and \eqref{eq:cascade-small-derivatives}, this yields
\[
 \lim_{p\to\infty}\int_{t_p}^{t_p+\delta}
 \|F(s)e^{B(s-t_p)}v_z\|^2ds=0.
\]
On the other hand, UO of the reduced pair gives, for every $p$,
\[
 \frac1\delta\int_{t_p}^{t_p+\delta}
 \|F(s)e^{B(s-t_p)}v_z\|^2ds\geq\mu_r\|v_z\|^2.
\]
The preceding limit therefore implies $v_z=0$. The measured-state trajectory then satisfies $\xi_p(t_p+s)=e^{A_os}v_\xi$, so \eqref{eq:cascade-small-output} reduces to
\[
 \int_0^\delta\|C_oe^{A_os}v_\xi\|^2ds=0.
\]
Since $(A_o,C_o)$ is observable, its Gramian on $[0,\delta]$ is positive definite, and hence $v_\xi=0$. We obtain $v=0$, contradicting $\|v\|=1$. This proves sufficiency.

\emph{Necessity.}
We construct a measured initial state for which the integral of $\|q\|^2$ is bounded above by a fixed multiple of the integral of $\|\Gamma z\|^2$. Define the constant left inverse
\[
 L:=(\mathcal O_-^\top\mathcal O_-)^{-1}\mathcal O_-^\top,
 \qquad L\mathcal O_-=I_m,
\]
and set $T(t):=L\mathcal R_-(t)$ and $\eta(t):=\xi(t)+T(t)z(t)$. The preceding regularity assumptions make $T$ continuously differentiable and uniformly bounded.

Let $S_-,S_+,S_0$ be the constant matrices selecting, respectively, blocks $0,\ldots,\nu-1$, blocks $1,\ldots,\nu$, and block $0$ from a stack of $\nu+1$ blocks of size $p$. Define
\[
 J:=C_oLS_--S_0,\qquad M:=LS_+-A_oLS_-.
\]
The identities
\[
 S_-\mathcal O_+=\mathcal O_-,\qquad
 S_+\mathcal O_+=\mathcal O_-A_o,\qquad
 S_0\mathcal O_+=C_o
\]
give $J\mathcal O_+=0$ and $M\mathcal O_+=0$. Because the rows of $N$ span the entire left nullspace of $\mathcal O_+$, there are constant matrices $\bar J,\bar M$ such that
\[
 J=\bar JN,\qquad M=\bar MN.
\]
This is the step that requires the full left nullspace in~\eqref{eq:cascade-nullspace}.

Stacking the recursions~\eqref{eq:general-recursion} gives
\[
 \dot{\mathcal R}_-(t)+\mathcal R_-(t)B
 =S_+\mathcal R_+(t)-\mathcal O_-G(t).
\]
Multiplying by $L$ and using $L\mathcal O_-=I_m$, we obtain
\[
 G(t)+\dot T(t)+T(t)B=LS_+\mathcal R_+(t).
\]
It follows that
\[
 \begin{aligned}
 \dot\eta(t)
 &=A_o\eta(t)+\bigl(LS_+-A_oLS_-\bigr)\mathcal R_+(t)z(t)\\
 &=A_o\eta(t)+\bar M\Gamma(t)z(t).
 \end{aligned}
\]
Moreover, $R_0(t)=0$ gives
\[
 C_oT(t)=\bigl(C_oLS_--S_0\bigr)\mathcal R_+(t)
         =\bar J\Gamma(t).
\]
Thus, with $f(t):=\Gamma(t)z(t)$, the transformed equations are
\begin{equation}
 \dot\eta(t)=A_o\eta(t)+\bar M f(t),\qquad
 q(t)=C_o\eta(t)-\bar J f(t).
 \label{eq:cascade-coordinate}
\end{equation}

Suppose the full system is UO with window $\delta$ and normalized lower bound $\mu$. Fix $t\geq0$ and $z_0\in\mathbb R^n$, and choose $z(t)=z_0$ and $\xi(t)=-T(t)z_0$. Then $\eta(t)=0$, so variation of constants in \eqref{eq:cascade-coordinate} gives
\[
 \begin{aligned}
 q(t+s)&=-\bar J f(t+s)\\
 &\quad+\int_0^s C_oe^{A_o(s-u)}\bar M f(t+u)du.
 \end{aligned}
\]
For $0\leq s\leq\delta$, Young's convolution inequality yields
\[
 \|q\|_{\mathcal L^2([t,t+\delta])}
 \leq K_\delta\|f\|_{\mathcal L^2([t,t+\delta])},
\]
where the constant
\[
 K_\delta:=\|\bar J\|+
       \int_0^\delta\|C_oe^{A_os}\bar M\|ds
\]
is finite and independent of $t$ and $z_0$. Set $c_\delta:=\max\{1,K_\delta^2\}$. Applying full UO to the chosen initial state gives
\[
 \begin{aligned}
 \mu\delta\|z_0\|^2
 &\leq\mu\delta\bigl(\|T(t)z_0\|^2+\|z_0\|^2\bigr)\\
 &\leq\int_0^\delta
       \bigl(\|q(t+s)\|^2+\|De^{Bs}z_0\|^2\bigr)ds\\
 &\leq c_\delta\int_0^\delta\|F(t+s)e^{Bs}z_0\|^2ds.
 \end{aligned}
\]
Hence the reduced pair is UO with the same window and normalized lower bound $\mu/c_\delta$.
\end{proof}

\end{document}